\documentclass[10pt]{amsart}

\usepackage{amsmath}
\usepackage{amsthm}
\usepackage{amsfonts}
\usepackage{amssymb,latexsym}
\usepackage[all]{xy}
\usepackage{graphicx}
\usepackage[mathscr]{eucal}
\usepackage{verbatim}
\usepackage{hyperref}

\theoremstyle{plain}
    \newtheorem{thm}{Theorem}[section]
    \newtheorem{prop}[thm]{Proposition}
    \newtheorem{lemma}[thm]{Lemma}
    
    \newtheorem{cor}[thm]{Corollary}

\theoremstyle{definition}

\theoremstyle{remark}

\numberwithin{equation}{section}

\begin{document}

\title{What is special about the Divisors of 12?}
\date{\today}

\author{Sunil K. Chebolu}
\address{Department of Mathematics \\
Illinois State University \\
Normal, IL 61790, USA} \email{schebol@ilstu.edu}

\author{Michael Mayers}
\address{Department of Mathematics \\
Illinois State University \\
Normal, IL 61790, USA} \email{mikemayers@gmail.com}

\maketitle

\thispagestyle{empty}
\section{Introduction}
This paper is a sequel to ``What is special about the divisors of 24?" \cite{24} in which the first author answered the following question which evolved out of a classroom discussion. For what values of $n$ does the multiplication table for $\mathbb{Z}_n$ have 1's only on the diagonal?  In other words,  when does $\mathbb{Z}_n$ have the property that whenever $ab =1$,  $a = b$.  It was shown that only the divisors of $24$ have this so-called diagonal property.   In fact, \cite{24} gives 5 proofs of this result.  In the last section of that paper, the following variation of this question was posed. For what values of $n$, does the multiplication table for $\mathbb{Z}_n[x]$ have 1's only on the diagonal? 
Note that this is a well-posed problem even though the multiplication table for  $\mathbb{Z}_n[x]$ has infinite size.  In fact, one can study this question over any ring $R$. Then the diagonal property for $R$ (the property of having 1's in the multiplication table for $R$ only on the diagonal, never off the diagonal) is equivalent to the following algebraic statement: every unit in $R$ is an involution. To see this, let $R$ be a ring with the diagonal property and let $a$ be a unit in $R$.  Then, by definition 
of a unit, there is an element $b$ in $R$ such that  $ab = 1$. Since $R$ has the diagonal property, $a = b$. This means $a^2 = 1$,  or equivalently, $a = a^{-1}$.  So $a$ is an involution.  For the other direction, suppose every unit in $R$ is an involution. If $ab = 1$,  then $a$ (being a unit) is also an involution. That is, $a^2=1$. Combining the last two equations, we get $ab=a^2$, and therefore $a = b$. This means $R$ has the diagonal property.

In this paper we will prove the following general result which gives a surprising answer to the above question.

\begin{thm}
For any positive integer $m$, the multiplication table for the polynomial ring $\mathbb{Z}_n[x_1, x_2, \cdots, x_m]$ has 1's only on the diagonal if and only if $n$ is a divisor of $12$. 
\end{thm}

We will use basic facts from commutative algebra which can be found in any standard textbook including \cite{AM} and \cite{DummitFoote}.

\section{Proof of the main theorem}
Fix an aribitrary positive integer $m$ which will be the number of variables in our polynomial ring. We begin with a straightforward observation.
Suppose $n$ is a positive integer for which $\mathbb{Z}_n[x_1, x_2, \cdots, x_m]$ has the diagonal property. Since $\mathbb{Z}_n$ is a subring of $\mathbb{Z}_n[x_1, x_2, \cdots, x_m]$, it follows that $\mathbb{Z}_n$  also has the diagonal property.  So $n$ has to be a divisor of 24; see \cite{24}.
 
$8$ and $24$ are the only numbers which divide 24 but not 12. So our main theorem will follow once we prove the following statements.

\begin{enumerate}
\item[(a)] $\mathbb{Z}_8[x_1, x_2, \cdots, x_m]$ and $\mathbb{Z}_{24}[x_1, x_2, \cdots, x_m]$ do not have the diagonal property.
\item[(b)] $\mathbb{Z}_n[x_1, x_2, \cdots, x_m]$ has the diagonal property when $n$ is a divisor of 24 that is not 8 or 24.
\end{enumerate}

\noindent
\underline{Proof of (a):} By the above discussion, it is enough to find, in both rings, a unit that is not an involution. To this end, we will use the following lemma which is well-known.
 
\begin{lemma} Let $R$ be a commutative ring. If $u$ is a unit and $r$ is a nilpotent element in $R$, then $u + r$ is a unit. 
\end{lemma}
\noindent
\begin{proof}
Let $k$ be the unique integer such that $r^k \ne 0$ and $r^{k+1} = 0$. Then an easy check shows that the formal inverse of $u + r$ is given by 
 $u^{-1}(1  - r/u + (r/u)^2 + \cdots + (-1)^k (r/u)^k )$.
\end{proof}

\noindent Returning to proof of (a):
\vskip 2mm \noindent
$\mathbb{Z}_8[x_1, x_2, \cdots, x_m]$: $2x_1$ is a nilpotent element in this ring and therefore, by the above lemma, $u = 1 + 2x_1$ is a unit. However, 
\[ u^2 -1 = (1 + 2x_1)^2 - 1=  4x_1 + 4x_1^2 \]
is not a multiple of $8$ in $\mathbb{Z}[x_1, x_2, \cdots, x_m]$, and hence not zero in $\mathbb{Z}_8[x_1, x_2, \cdots, x_m]$. Therefore $u$ is not an involution.

\noindent
$\mathbb{Z}_{24}[x_1, x_2, \cdots, x_m]$: This is done similarly. $6x_1$ is a nilpotent element, and therefore $u = 1 + 6x_1$ is a unit. But 
\[u^2 -1 = (1 + 6x_1)^2 - 1 = 12x_1 + 36x_1^2\]
is not a multiple of 24 in $\mathbb{Z}[x_1, x_2, \cdots, x_m]$, and hence not zero in $\mathbb{Z}_{24}[x_1, x_2, \cdots, x_m]$. Therefore $u$ is not an involution. This completes proof of part (a).

\vskip 2mm
\noindent
\underline{Proof of (b):} Here we will use the following commutative algebra result which gives a characterisation of units in polynomial rings. Although this proposition is well-known, it is often stated only in the one-variable case.

\begin{prop} Let $R$ be a commutative ring. A polynomial $f(x_1, x_2, \cdots, x_m)$ is a unit in $R[x_1, x_2, \cdots, x_m]$ if and only if the constant term of $f$ is a unit in $R$ and all other coefficients of $f$ are nilpotent elements in $R$.
\end{prop}

\begin{proof} The proof of the ``if direction"  follows easily from the above lemma and induction.
As for the ``only if direction," we first prove it in the special case when $R$ is an integral domain. To this end, suppose $R$ is an integral domain and let $f$ (as above) be a unit in $R[x_1, x_2, \cdots, x_m]$. Then there exists a polynomial $g$ such that $fg = 1$. Since $R$ is an integral domain, we have $\deg(fg) = \deg(f) + \deg(g)$.  So by comparing degrees we get 
\[ \deg(f) + \deg(g) = 0.\]
This means $f$ and $g$ are constant terms. Since $fg=1$, it follows that $f$ is a  unit in $R$.

For the general case,  let $f$ be a unit in $R[x_1, x_2, \cdots, x_m]$ and consider the ring homomorphism 
\[ R[x_1, x_2, \cdots, x_m]  \longrightarrow  R,  \]
which reduces modulo the ideal $(x_1, x_2, \cdots, x_m)$. The image of $f$ under this homomorphism is the constant term of $f$. Since every ring homomorphism sends units to units, we conclude that the constant term of $f$ is a unit.
To see that the other coefficients of $f$ are nilpotents in $R$,
we let $p$ be any prime ideal in $R$, and consider the ring homomorphism 
\[ R[x_1, x_2, \cdots, x_m]  \longrightarrow R/p[x_1, x_2, \cdots, x_m] \]
which reduces the coefficients of a polynomial modulo $p$.    Since $R/p$ is an integral domain, by the special case proved above, we conclude that the image of $f$ in $ R/p[x_1, x_2, \cdots, x_m]$  is a constant, i.e., the coefficients of all higher degree terms are zero. This means they all belong to $p$.  Since the choice of prime ideal $p$
was arbitrary, it follows that all  the coefficients of $f$, other than the constant term, belong to the intersection of all prime ideals. The latter  is precisely the nilradical of $R$ (see \cite{AM}), and therefore the coefficients in question are nilpotents.

\end{proof}

Returning to the proof of (b): we will show that whenever $n$ is divisor of 24 other than 8 or 24, every unit  $u$ in $\mathbb{Z}_n[x_1, x_2, \cdots, x_m]$ is an involution, i.e., $u^2 - 1 = 0$. The values of $n$ to be considered are $2, 3, 4, 6$ and 12. 

The rings $\mathbb{Z}_2$, $\mathbb{Z}_3$, and $\mathbb{Z}_6$ are reduced. That is, they do not have any non-zero nilpotent elements. Therefore, by the above proposition, the units in $\mathbb{Z}_2[x_1, x_2, \cdots, x_m]$, $\mathbb{Z}_3[x_1, x_2, \cdots, x_m]$, and $\mathbb{Z}_6[x_1, x_2, \cdots, x_m]$ are exactly those in $\mathbb{Z}_2$, $\mathbb{Z}_3$, and $\mathbb{Z}_6$ respectively. Since the rings $\mathbb{Z}_2$, $\mathbb{Z}_3$, and $\mathbb{Z}_6$  have the diagonal 
property, so do the corresponding polynomial rings. This leaves us with $4$ and $12$.

\noindent
$\mathbb{Z}_4[x_1, x_2, \cdots, x_m]$: $2$ is the only non-zero nilpotent and $1$ and $-1$ are the only units in $\mathbb{Z}_4$. Therefore every unit $u$ in $\mathbb{Z}_4[x_1, x_2, \cdots, x_m]$ is of the form $2h-1$ or $2h+1$, where $h$ is an arbitrary polynomial. Then $u^2 = 1 \mod{ 4}$  in $\mathbb{Z}[x_1, x_2, \cdots, x_m]$, and therefore $u$ is an involution in $\mathbb{Z}_4[x_1, x_2, \cdots, x_m]$.

\noindent
$\mathbb{Z}_{12}[x_1, x_2, \cdots, x_m]$: $6$ is the only non-zero nilpotent in $\mathbb{Z}_{12}$. Therefore every unit $u$ in $\mathbb{Z}_{12}[x_1, x_2, \cdots, x_m]$ is of the form $u = 6h + r$, where $h$ is an arbitrary polynomial and $r$ is a unit in $\mathbb{Z}_{12}$. 
Now, $u^2 = 36 h^2 + 12rh + r^2$. The latter is equal to $r^2$ in $\mathbb{Z}_{12}[x_1, x_2, \cdots, x_m]$. Since $\mathbb{Z}_{12}$ has the diagonal 
property, we have $r^2 = 1 \mod{12}$. Therefore $u$ is an involution in $\mathbb{Z}_{12}[x_1, x_2, \cdots, x_m]$, as desired.

\end{document}